\documentclass{amsart}
\usepackage{latexsym,amssymb,amsmath}
\usepackage{longtable}
\usepackage{graphics}
\usepackage{pstricks}

\theoremstyle{definition}

\newtheorem{stat}{Proposition}[section]

\newtheorem{exam}{Example}
\newtheorem{defn}{Definition}[section]
\theoremstyle{remark}

\numberwithin{equation}{section} 

\begin{document}

\title{On One Function Defined on the Cartesian Product and Guinness Numbers}
\author{R.A.Zatorsky}
\begin{abstract}
New numbers, called Guinness numbers, are introduced using certain function of natural argument. Few problems related to these numbers are formulated.
\end{abstract}


\maketitle

''The year 2000 was approaching. I realized humans living on the Earth were killing It with their "civilization" and "civility" and finally would perish themselves like bacteria in a dying human organism. But I rejoiced at our fate to live to see the turn into the new millennium. As my destiny is to be involved with mathematics, I felt a burning desire to "exalt" the year 2000 using the language of mathematics. Right then I thought up the mathematical idea described below, which to some extent made it possible to achieve the desired goal. I saw that the number $G_{2000}$ is the one-half Guinness number, but could not publish it due to some reasons. Then, due to similar reasons, I missed the successive one-half Guinness numbers $G_{2001}, G_{2006}$ and $G_{2007}.$ But I cannot but record the next one-half number Guinness $G_{2013}$ for I do not know whether I will see the year 2031 when the new one-half Guinness number $G_{2031}$ "appears"'' --- this is how the preface to the author's book ''Guinness Number $G_{2013}$'' begins. Besides the preface, this book includes only one number, the so-called one-half Guinness number $G_{2013}.$

\begin{center}
\resizebox{0.49\textwidth}{!}{\includegraphics{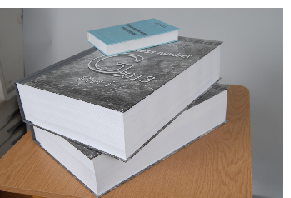}}\hfill
\resizebox{0.49\textwidth}{!}{\includegraphics{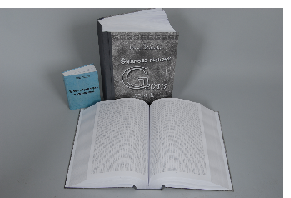}}\nopagebreak

\emph{This photo presents the book ''Guinness Number $G_{2013}$'', which consists of two parts. As a comparison, next to it is the author's book of standard size containing 508 pages.}
\end{center}

 The Guinness number $G_{2013}$ consists of 40,259,996 digits and is the smallest positive integer\footnote{April 04, 1978, at the scheduled meeting of the Academic Council of the Computing Center at the AS of the USSR, which was held to discuss the next thesis for a candidate degree in physics and mathematics, the question was raised whether the submitted thesis qualified for a candidate degree in physics and mathematics or rather a candidate degree in science. The thesis contained no theorem; there was only one program written in LISP, which allowed to determine the operations of artificial intelligence. The opinions of the present council members divided. The director of the institute A.Dorodnitsin had the say. He stated that \emph{each program can be considered a theorem on the solubility of some problem. And if the program works, then one can claim that the theorem is proved.}
This statement was supported by A.Markov who suggested that the statement be specified: \emph{"Each program is a theorem, the validity of which is proved for those cases when the program gives the correct answer."} Thus, to prove the fact that this number is the smallest positive integer with the described above properties, one is to develop a simple program and the computer is to produce the expected result.} having the following nontrivial properties: the first four figures of the number $G_{2013}$ make up the number 2013. If these first four figures are removed and added to the end of the number $G_{2013}$ in the same order, the formed number will be reduced by 2013 times. The book consists of two parts with the total number of pages - 3655. Each page, except the last one, contains 106 lines with 104 signs in each of them.

Following is the definition of the whole and one-half Guinness numbers and the special function $f_n,$  with the help of which these numbers can be generated. At the end of the article you can find a number of questions I failed to find the answer to as well as some useful applications of the function $f_n.$

\section{Operation $f_n$ and Guinness Numbers}
\begin{verse}
\emph{Let's call numbers by People's name.\\
People will remember numbers as long as they exist.\\
Numbers will "remember" people in case the latter disappear.}
\end{verse}

Let us have some positive $k$-digit number $n$ and assign the Cartesian product to each such a positive number
 $$\Omega_n=\{0,1,2,\ldots, 10^k-1\}\times
\{0,1,2,\ldots,n-1\}.$$
This Cartesian product can be presented graphically on the Cartesian plane as a rectangle of integral points (points with integer coordinates) with the main diagonal connecting the beginning of the coordinates with the point $(10^k-1,n-1)$.

\begin{defn}
\emph{Let the function $f_n:\Omega_n\rightarrow\Omega_n$ assign the point $(x',y')\in\Omega_n$ to the arbitrary point  $(x,y)\in\Omega_n$, i.e.  $f_n(x,y)=(x',y')$, while the following equality holds}
 $$nx+y=x'+10^ky'.$$
\end{defn}

\emph{The function inverse} to the function $f_n$ is the function $f_n^{-1}$, when the equality $f^{-1}_n(x',y')=(x,y)$  is performed as long as $f_n(x,y)=(x',y').$

	It is easy to prove that both functions defined above are bijective mappings of the set $\Omega_n$ into themselves.
	The point $(x,y)$  is \emph{the fixed point of the function} $f_n,$ provided that $f_n(x,y)=(x,y)$. It is easy to prove that the coordinates of the points of the function $f_n$ satisfy the equality
$$(n-1)x=(10^k-1)y,$$ and these points are on the main diagonal of the Cartesian rectangle $\Omega_n$ , while for an arbitrary positive $n$ the function $f_n$ has at least two fixed points $(0,0)$ and $(10^k-1,n-1).$

\begin{exam}
\label{exam1} All the fixed points of the function $f_{34}$ in the set $\Omega_{34}$ are defined through the equation $$x=3y.$$ They are $34$ points: $(3i,i),\,\,i=0,1,\ldots,33.$
\end{exam}

As the set $\Omega_n$ is bounded and consists only of fixed and moving points, it is arguable that with the help of some initial point $(x,y),$ which is called \emph{the generating point}, the function $f_n$ generates \emph{the orbit}
$$O_n(x,y)=\{(x,y),f_n(x,y),f_n^2(x,y),\ldots,f_n^{r-1}\}$$
of length $r$, where $r$ is the smallest positive integer, for which the following equality holds $f_n^r(x,y)=(x,y)$.

Below among all the generating points $(x,y)$ of the orbit $O_n(x,y),$ generated by the function $f_n,$ the point $(n,0)$ is of special importance and is called \emph{a standard generating point} of the set $\Omega_n.$

\begin{defn}
\label{def4}\emph{The point
$\overline{(x,y)}=(\overline{x},\overline{y})$} is \emph{the conjugate point}
to the point $(x,y),$ if the following equalities hold
\begin{equation}\label{6}
    \overline{x}=10^k-1-x,\,\,\overline{y}=n-1-y.
\end{equation}
\end{defn}

It is easy to show that the conjugate points are symmetric relative to the point, which is the symmetry center of the square $\Omega_n.$

\begin{stat}\label{stat.f.spr}
\emph{For the arbitrary point $(x,y)\in\Omega_n$ from the equality \begin{equation}\label{8}
    f_n(x,y)=\overline{(x,y)}
\end{equation}
follows the equality
\begin{equation}\label{9}
    f_n\overline{(x,y)}=(x,y)
\end{equation} and vice versa.}
\end{stat}
Thus, sometimes conjugate points form the orbit of length 2.

\begin{exam}
In the set $\Omega_n,$ where $n$ is a single-digit or double-digit number, there are no self-conjugate orbits of length $2$; in the set $\Omega_{103}$ there are $6$ different self-conjugate orbits of length $2:$ $$(76,95)\leftrightarrow(923,7), (153,87)\leftrightarrow(846,15), (230,79)\leftrightarrow(769,23),$$$$ (307,71)\leftrightarrow(692,31), (384,63)\leftrightarrow(615,39), (461,55)\leftrightarrow(538,47),$$
and in the set $\Omega_{142}$ they number $71.$
\end{exam}

\begin{stat}
\label{stat5} \emph{Given two non-conjugate points $(x,y), (x',y')$ of the set $\Omega_n,$ the equality
$$f_n(x,y)=(x',y'),$$ implies the equality
$$f_n\overline{(x,y)}=\overline{(x',y')}.$$}
\end{stat}
\begin{proof} The first equality of this proposition equals the equality
$$nx+y=10^ky'+x'.$$ Taking it into consideration, we have the following equalities
$$n(10^k-1-x)+n-1-y=10^kn-1-(nx+y)=10^k(n-1-y')+(10^k-1-x'),$$
which are equal to the second equality of this proposition.\end{proof}

Thus, if among the points of the orbit $O_n(x,y)$, generated by the function $f_n$, there is no pair of conjugate points, then in the Cartesian set $\Omega_n$ there is one more orbit $O_n(\overline{x},\overline{y}),$  which is called the conjugate orbit to the orbit $O_n(x,y)$, while graphically these orbits are symmetric relative to the symmetry center of the rectangle $\Omega_n$. There are also self-conjugate orbits containing interconjugate points only. If presented graphically, these orbits are centrosymmetric.

\begin{center}\psset{linewidth=0.4pt, arrowsize=3pt 2, arrowlength=1.5}
\begin{pspicture}(0,0)(3.3,2.2)
\pscircle(0.3,0.3){0.3} \pscircle(3,0.3){0.3}
\pscircle(0.3,1.9){0.3} \pscircle(3,1.9){0.3}
\psline[linestyle=dashed](0.3,0.6)(0.3,1.6)
\psline[linestyle=dashed](3,0.6)(3,1.6)
\psline{->}(0.6,0.3)(2.7,0.3) \psline{->}(0.6,1.9)(2.7,1.9)
\psline{->}(2.79,0.51)(0.51,1.69)
\psline{->}(2.79,1.69)(0.51,0.51)
\end{pspicture}\nopagebreak

Fig.2
\end{center}

In this figure, the dashed segments connect the pairs of the conjugate points, and $(x,y)\rightarrow (x',y')$ means $f_{37}(x,y)=(x',y').$ The set $\Omega_{37}$ numbers $34$ orbits with such a structure, and the set $\Omega_{46}$ --- $18$ ones.

\begin{defn}	\emph{The whole Guinness number is a number $G_n$, generated by the function $f_n$ with the help of a standard generating point $(n,0)$, if the length of its respective orbit $O_n(n,0)$ equals $|\Omega_n|-2$.}
 \end{defn}

 The one-half Guinness number is defined in the same way.
 \begin{defn} \emph{The one-half Guinness number, generated by the function $f_n$ with the help of a generating point $(n,0),$ is a number $G_n$, the respective orbit length of which equals $\frac{1}{2}|\Omega_n|-1$.}
 \end{defn}

We cite the program in the Maple language
$n := 2013: k := 4: x := 2013: y := 0: i := 0: z := x: t := y: m := n*z+t:t := floor(m/10^k): z := m-10^k*t: i := i+1: while `or`(x <> z, y <> t) do  m := n*z+t: t := floor(m/10^k): z := m-10^k*t: i := i+1 end do: print(i)$

In this program, at the input we have $k=4$-digit number $n=2013$ and the standard generating point $(n,0)=(2013,0),$ and at the output we have the orbit length $O_{2013}(2013,0)=i.$ Under this program, the computer produces the number $10064999$ after no longer than $5$-minute work. Since all the first components of each pair of the orbit is four-digit numbers, then the one-half Guinness number $G_{2013}$ contains only $$10064999\cdot 4=40259996$$ figures.

It is easy to make certain that among single-digit numbers, the one-half Guinness number is the number $9$, and the whole Guinness numbers are the numbers $2,3$ and $6$.

Let us cite all double-digit one-half Guinness numbers:
$$ 14,  \textbf{20,  21},  24,  27,  30,  33,  41,  48,  51,  54,  62,  66,  69,  75,  77,  87,  90,  92.$$
The following are all three-digit one-half Guinness numbers:
$$ 102 , 105 , 108 , 135 , 144 , 162 , 165 , 183 , 189 , 192 , 204 , 213 , 222 , 231 ,
240, 261,$$
 $$267 , 273 , 276 , 291 , 294 , 303 , 306 , 309 , 327 , 330 , 339 , 357 , 372 , 378 ,
 390,
 420,$$
 $$444 , 456 , 465 , 474 , 498 , 507 , 513 , 522 , 525 , 534 , 537 , 543 , 564 , 567 ,
 585,
 588,$$
 $$600 , 603 , 609 , 612 , 621 , 639 , 645 , 660 , 663 , 669 , 672 , 696 , 705 , 726 ,
 732,
 738,$$
 $$765 , 774 , 789 , 795 , 807 , 819 , 822 , 834 , 840 , 855 , 873 , 885 , 891 , 894 ,
 906,
 921,$$
 $$933 , 936 , 942 , 957 , 975 , 981 , 990.$$

 \begin{defn}	\emph{Two one-half Guinness numbers $G_n$ and $G_{n+1}$ are called \textbf{twin one-half Guinness numbers}.}
\end{defn}
Among the double-digit numbers, there is only one pair of the twin one-half Guinness numbers, and among the three-digit numbers there is no pair like that. The following are first few pairs of four-digit twin one-half Guinness numbers:
$$(1085,1086),\,\,(1091,1092),\,\,(1109,1110),\,\,(1160,1161),\,\,(1187,1188),
\,\,(1208,1209),$$ $$(1316,1317),\,\,
(1337,1338),\,\,(1370,1371),\,\,(1553,1554),\,\,(1658,1659),\,\,(1742,1743)
,$$
$$(1775,1776),\,\,(1796,1797),\,\,(1889,1890),\,\,(1922,1923),\,\,
(2000,2001),\,\,(2006,2007),$$
$$(2174,2175),\ldots.$$

	If the function $f_n$  in the set $\Omega_n$ generates $m_1$ orbits of the length $r_1,$  $m_2$ orbits of the length $r_2$ etc. $m_s$ orbits of the length $r_s,$ this fact will be denoted by
 $$\Omega_n \sim \{r_1^{m_1},r_2^{m_2},\ldots,r_s^{m_s}\},$$ while the following equality holds $|\Omega_n|=\sum_{i=1}^sr_im_i.$

Thus, decomposition of the Cartesian set $\Omega_n,$ connected with a positive integer $n$, by the orbits, to a great extent, resembles factorization of positive integers.
$$\Omega_{10} \sim \{1^{10},3^{330}\},\,\,
\Omega_{11}\sim \{1^2,3^2,39^{28}\},\,\,
 \Omega_{12}\sim \{1^{12},54^{22}\},\,\,
\Omega_{13}\sim \{1^4,216^6\}$$
$$\Omega_{14}\sim \{1^2,699^2\},\,\,
\Omega_{15}\sim \{1^2,107^{14}\},\,\,
\Omega_{16}\sim \{1^4,3^{12},5^{24},15^{96}\},\,\,
\Omega_{17}\sim \{1^2,283^6\}$$
$$\Omega_{18}\sim \{1^2,3^2,128^2,384^4\},\,\,
\Omega_{19}\sim \{1^{10},15^{126}\},\,\,
\Omega_{20}\sim \{1^2,999^2\},\,\,
\Omega_{21}\sim \{1^2,1049^2\}.$$

Due to this, it is natural that the following questions arise.

	\textbf{Problem 1.} \emph{Are there whole Guinness numbers $G_n$ for multi-digit numbers $n$?}

	\textbf{Problem 2.}  \emph{The cardinality of the set of the one-half Guinness numbers should be studied.}

	\textbf{Problem 3.} \emph{The law of distribution of positive integers $n$ in the natural sequence, for which numbers $G_n$ are one-half Guinness numbers, should be studied.}

	\textbf{Problem 4.} \emph{The cardinality of the set of twin Guinness numbers should be studied.}

\section{Operation $f_n$ and Multidigit Number Product}
\begin{verse}
\emph{ "The external is similar to the internal; the little is the same as the big; the law is one for all... " (Hermes)}
\end{verse}
 For simple positive integers $n,$
the operation $f_n$ is found in every algorithm for multidigit number multiplication.
\begin{exam}\emph{Suppose it is necessary to multiply a single-digit number $n=2$ by some positive integer $\overline{x_4x_3x_2x_1x_0}=72389.$ The product of these numbers can be obtained with the operation $f_2,$ which is sequentially applied to the pairs
$$(x_0,y_0)=(9,0),\,(x_1,y_1)=(8,1),\,(x_2,y_2)=(3,1),\,
(x_3,y_3)=(2,0),\,(x_4,y_4)=(7,0).$$ At that we get the respective pairs:
$$(x'_0,y'_0)=(8,1),\,(x'_1,y'_1)=(7,1),\,(x'_2,y'_2)=(7,0),\,
(x'_3,y'_3)=(4,0),\,(x'_4,y'_4)=(4,1).$$  Note that according to the multiplication algorithm, we start from the zero value $y_0,$ and all other values $y_i,\,i=1,2,3,4$ are selected according with the equalities $y_i=y'_{i-1},\,i=1,2,3,4.$ The product result is the number}
$$\overline{y'_4x'_4x'_3x'_2x'_1x'_0}.$$
\end{exam}
The operation $f_n$ generalizes the algorithm for multiplication of a single-digit number by an arbitrary positive $k$-digit number. Let us have some multidigit number $m$ and $k$--digit number $n.$ In order to find the product $mn$, it is necessary to partition a number $m$, starting from its end, into $s$ groups with $k$ figures in each one, and then to proceed as in the previous example.
\begin{exam}\emph{If, for instance, $m=2345678$ and $n=23,$
then in this case, the operation $f_{23}$ is applied sequentially to the following pairs
$$(x_0,y_0)=(78,0),\,(x_1,y_1)=(56,17),\,(x_2,y_2)=(34,13),\,
(x_3,y_3)=(2,7).$$ At that we obtain the respective pairs
$$(x'_0,y'_0)=(94,17),\,(x'_1,y'_1)=(05,13),\,(x'_2,y'_2)=
(95,7),\,(x'_3,y'_3)=(53,0)$$ and the product result is the number}
$\overline{0'53'95'05'94}.$
\end{exam}
Thus, we obtain the algorithm for multidigit number multiplication.

Now generation of the Guinness numbers generated by a $k$--digit positive integer $n$ and a standard generating point $(n,0)$ is obvious. To generate this number according to the described above algorithm for multidigit numbers, it is enough to write the first components of the orbit $O_n(n,0)$ from the right to the left in succession.

\section{Operation $f_n$ and Tilings}

\begin{verse}
\emph{Numbers like people have their ''face'' and ''nature''.}
\end{verse}

Since the operation $f_n$ for some $n$ partitions the set $\Omega_n$ into an even number of interconjugate orbits and a certain number of self-conjugate orbits, then with the help of this operation one can construct a rectangular tiling with the symmetry center at the intersection of the diagonals of this rectangular. For this, the points of the interconjugate orbits should be combined and filled with some color.
\begin{exam}
\label{exam4} $f_8$--operation partitions the set $\Omega_8$ into six orbits with the length of $13$ and two fixed points. Combining the sets of the interconjugate orbit pairs into one set, we get three groups with $26$ points in each one and one group consisting of two fixed points. Now we shall relate some color to each group of points, and the unit square on the Cartesian plane to each pair of these sets. Thus, we obtain a colored rectangular. Paste together several obtained colored rectangulars and obtain a centrally symmetric tiling, which is given by the positive integer $8.$

\begin{center}
\resizebox{0.75\textwidth}{!}{\includegraphics{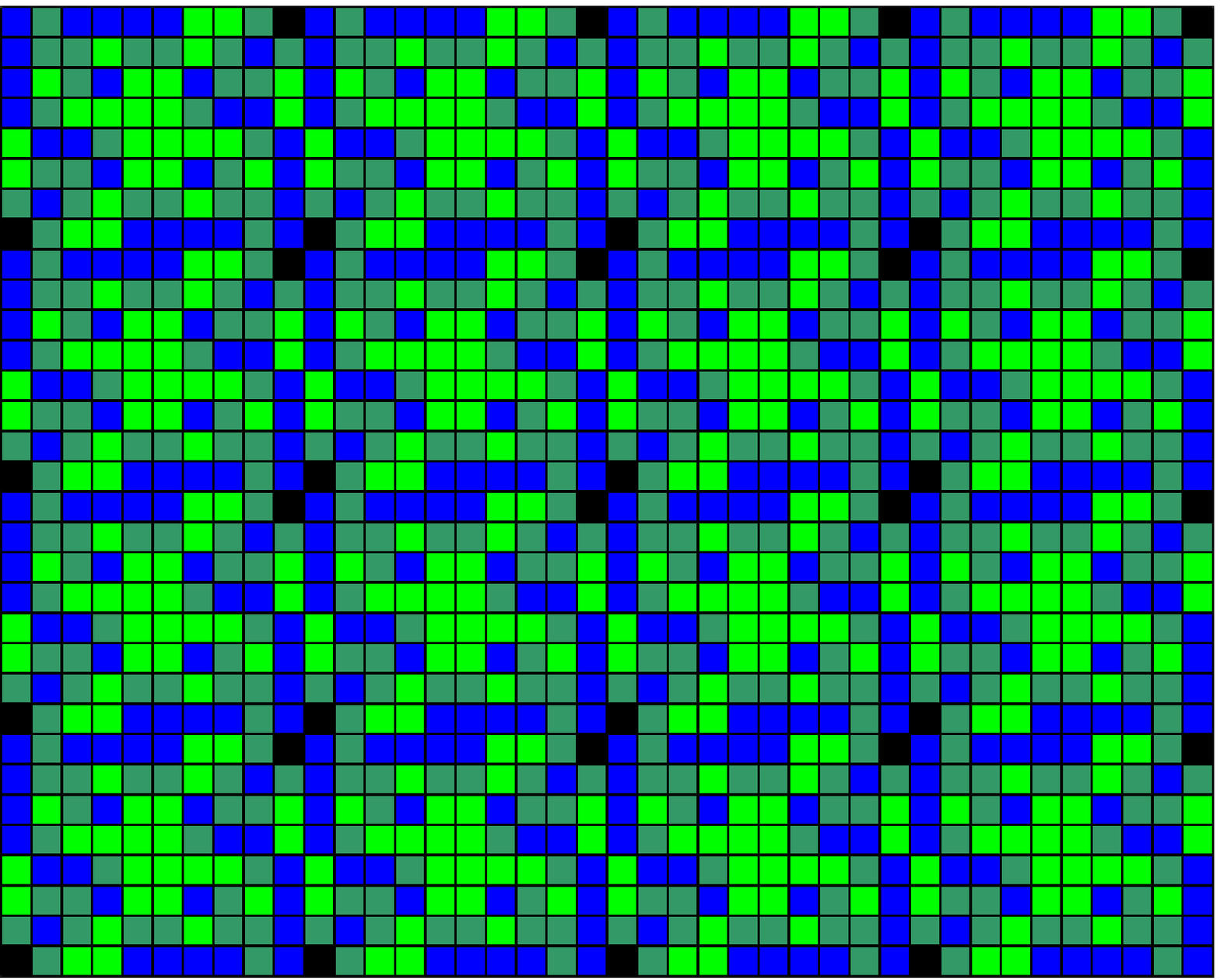}}\nopagebreak

Fig.3.
\end{center}

\end{exam}

Note that the distinctive pattern corresponds to the representation on the Cartesian plane of points relating to the orbits $O_n(n,0).$ Thus, each positive integer has its ''face'' and ''nature''. Below in Fig.4 and 5 the points of the orbits $O_9(9,0)$ and $O_7(7,0)$ are represented respectively.

\begin{center}
\resizebox{0.49\textwidth}{!}{\includegraphics{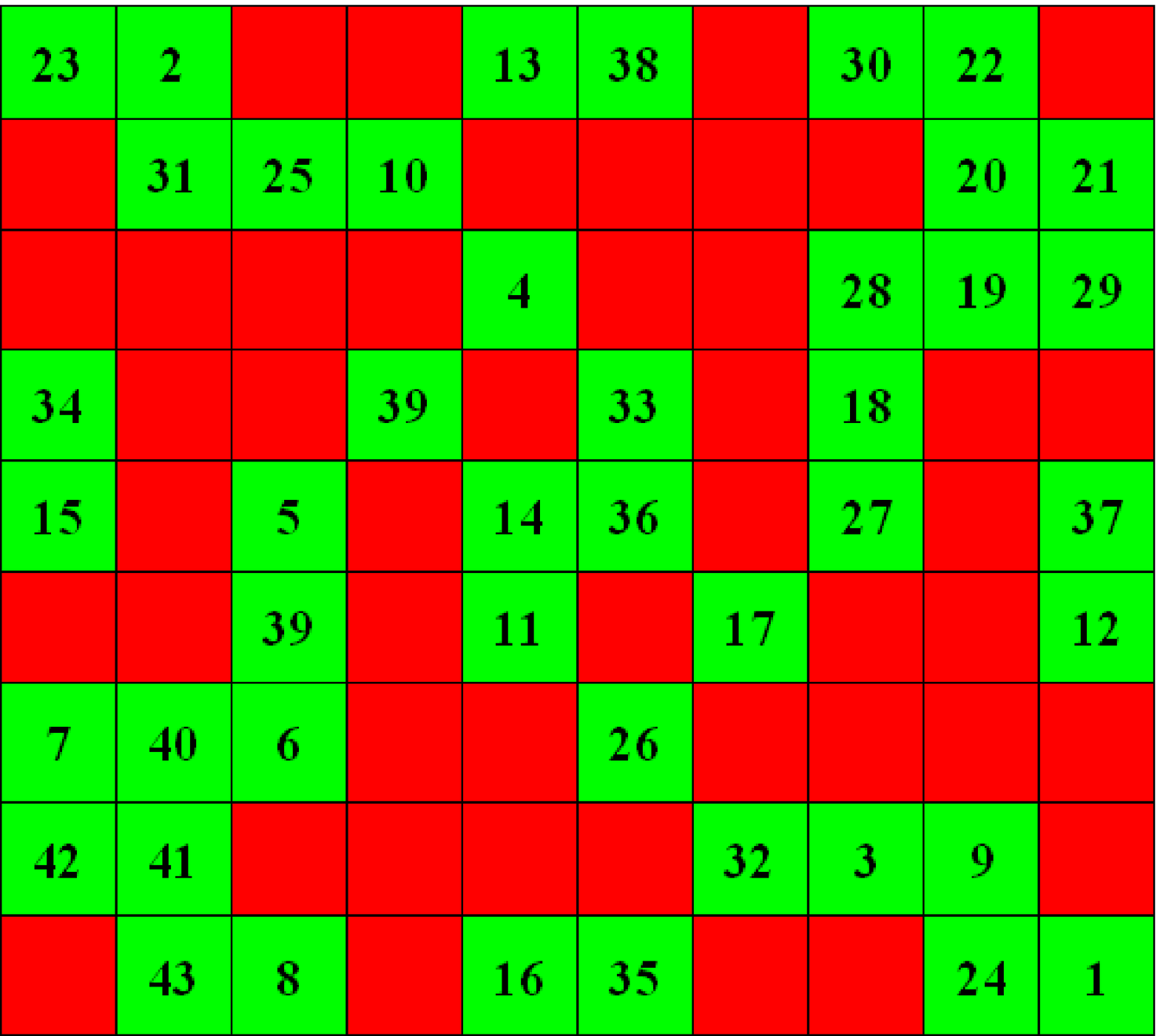}}\hfill
\raisebox{0.049\textwidth}{\resizebox{0.49\textwidth}{!}{\includegraphics{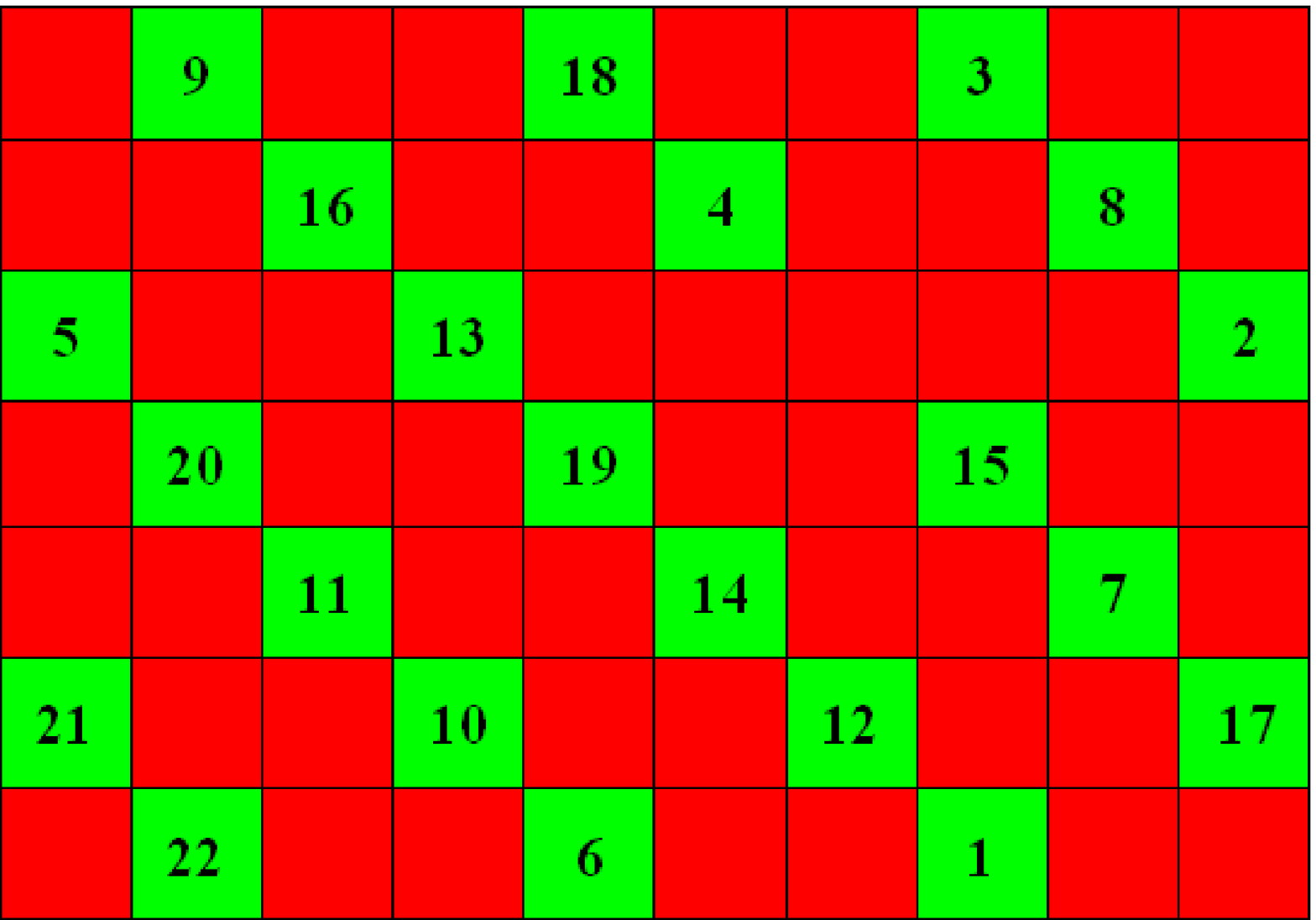}}}\nopagebreak

Fig.4.\hspace{54mm}  Fig.5.
\end{center}

\section{Guinness Numbers and Pseudorandom Number Generators}

\begin{quotation}\emph{
"The generation of random numbers is too important to be left to chance." (Robert R. Coveyou)}
\end{quotation}

The one-dimensional pseudorandom number generator (PRNG)\cite{Zub} stands for the algorithm generating the number sequence, the elements of which are almost independent of each other and evenly distributed on some segment.

PRNG is applied in various fields of human knowledge: computer science, programming, Monte Carlo method
\cite{Sobol}, cryptography \cite{Verb} etc.

 The first algorithm for generating pseudorandom numbers was offered by the American mathematician, one of the founders of computer science, John von Neumann. His algorithm is also known as the middle-square method consisting in choosing an arbitrary 4--digit number $a=0,a_1a_2a_3a_4;$ squaring it $a^2=0,a_1'a_2'a_3'a_4'a_5'a_6'a_7'a_8'$ and passing on to a new 4--digit number $0,a_3'a_4'a_5'a_6'$ etc.

The main advantage of this arithmetic algorithm is its simplicity, and its disadvantage is that it generates no more than 10000 different numbers.

In \cite{MMS}, the first author of this article has offered a two-dimensional arithmetic generator of pseudorandom numbers, which is based on the operation $f_n.$ This PRNG is somewhat similar to John von Neumann's arithmetic generator. The second author of the article \cite{MMS} has tested the offered generator in the system of distance learning and knowledge control, and the third author has tested it with the help of standard criteria: uniformity, intervals, "maximum--t" and poker criterion.

 In this article, the sequence of pseudorandom numbers is given by the first and second components of the orbit $O_n(n,0).$ In this context a positive integer $n$
is chosen so that the number $G_n$ is a whole or one-half Guinness number. For instance, with the help of the one-half Guinness number $G_{200000}$, we can construct an orbit, which is a two-dimensional random array with the period of
$99999999999.$  Despite John von Neumann's statement: "anyone who attempts to generate random numbers by deterministic means is, of course, living in a state of sin," the aforementioned generator has an enormous period and has been successfully tested.

 If, for instance, all the points of the orbit $O_{2000}(2000,0)$ are laid out on an A4 sheet of paper, they will cover it so that the sheet is solid grey at different intervals of time. Hence it follows that the aforesaid PRNG is also applicable to the Monte Carlo method.

One of the most accepted algorithms for generating pseudorandom numbers $x_n,\,\,n=0,1,2,\ldots,L$ is the algorithm offered by the American mathematician Derrick Henry Lehmer. This algorithm is given by the following equalities:
$$m_{n+1}\equiv5^{17}m_n(mod\, 2^{40}),\,\,n=0,1,2,\ldots,L,\,\,m_0=1,$$
$$x_n=2^{-40}m_n.$$

As for modern PRNG, Mersenne twister is widely applied. This generator was proposed in 1997 by Matsumoto and Nishimura. Its positive quality is its enormous period $(2^{19937}-1)$ and even distribution in 623 dimensions. But this generator cannot be applied in cryptography because there is an algorithm which identifies the sequence generated with the help of the Mersenne twister as nonrandom.

Let us point out one more application of the function $f_n.$
	 For rather high values $n$, to each generating point $(x,y)$ the function $f_n$ associates some array of numbers or its fragment, under which it is difficult to reproduce numbers $(n,x,y)$. It means the function $f_n$ is a difficult-to-invert function and can be used when generating key words for the one-time pad method.

\textbf{Problem 5.} \emph{The cryptosystem with the public key should be built based on the function $f_n$.}

	The author hopes that in future the function $f_n$ will be productively applied in a new way and will take the rightful place in the theory of numbers.

\end{document}